\documentclass{article}

\usepackage{graphicx}
\usepackage{url}
\usepackage{color}
\usepackage{amssymb,amsmath,amsthm}

\newtheorem{theorem}{Theorem}

\newtheorem{lemma}[theorem]{Lemma}
\newtheorem{corollary}[theorem]{Corollary}

\theoremstyle{definition}
\newtheorem{definition}[theorem]{Definition}

\theoremstyle{remark}
\newtheorem*{remark}{Remark}

\DeclareMathOperator{\KP}{\textit{K}\,}
\DeclareMathOperator{\KS}{\textit{C}\,}
\DeclareMathOperator{\KPt}{\textit{K}}

\DeclareMathOperator{\MLR}{\textit{MLR}}
\DeclareMathOperator{\m}{\mathbf{m}}

\newcommand{\leK}{\mathrel{\le_{\textrm{LK}}}}
\newcommand{\leMLR}{\mathrel{\le_{\textrm{LR}}}}
\newcommand{\cnd}{\,|\,}
\let\le=\leqslant
\let\ge=\geqslant

\title{$\KPt$-trivial, $\KPt$-low and $\MLR$-low sequences:\\
a tutorial}
\author{Laurent~Bienvenu\thanks{LIAFA, CNRS, University Paris Diderot, Poncelet lab (Moscow),\hfil\break \texttt{laurent.bienvenu@computability.fr}},~ Alexander~Shen\thanks{LIRMM, Montpellier, CNRS, UM2, on leave from IITP RAS, Moscow,\hfil\break
\texttt{alexander.shen@lirmm.fr}}}
\date{}
\begin{document}
\maketitle

\begin{abstract}
A remarkable achievement in algorithmic randomness and algorithmic information theory was the discovery of the notions of $\KPt$-trivial, \hbox{$\KPt$-low} and Martin-L\"of-random-low sets: three different definitions turn out to be equivalent for very non-trivial reasons~\cite{dhns,nies-main,hs}. This survey, based on the course taught by one of the authors (L.B.) in Poncelet laboratory (CNRS, Moscow) in 2014, provides an exposition of the proof of this equivalence and some related results.

We assume that the reader is familiar with basic notions of algorithmic information theory (see, e.g., \cite{uppsala-notes} for introduction and \cite{usv} for more detailed exposition). More information about the subject and its history can be found in~\cite{nies,dh}.
\end{abstract}

\section{Notation}
We consider the Cantor space $\mathbb{B}^\mathbb{N}$ of infinite binary sequences $a_0a_1\ldots$; points in this space are idendified with sets (each sequence is considered as a characteristic sequence of a set of natural numbers) or paths in the full binary tree (nodes are elements of $\mathbb{B}^*$, i.e., binary strings; a sequence $a$ is a path going through its prefixes $(a)_n=a_0a_1\ldots a_{n-1}$). We denote plain Kolmogorov complexity by $\KS(x)$; we use $\KS(x,y)$ to denote complexity of pairs and $\KS(x\cnd y)$ for conditional complexity. The arguments $x,y$ here are binary strings, natural numbers (that are often identified with binary strings using a standard bijection) or some other finite objects. Similar notation with $\KP$ instead of $\KS$ is used for prefix complexity.  

 By $\m(x)$ we denote the discrete a priori probability of $x$, the largest lower semicomputable semimeasure on $\mathbb{N}$; it is equal to $2^{-\KPt(x)}$ up to a $\Theta(1)$-factor.  The same notation is used when $x$ is a binary string (identified with the corresponding natural number) or some other finite object.

\section{$\KPt$-trivial Sets: Definition and Existence}

Consider an infinite bit sequence and complexities of its prefixes. If they are small, the sequence is computable or almost computable; if they are big, the sequence looks random.   This idea goes back to 1960s and appears in algorithmic information theory in different forms (Schnorr--Levin criterion of randomness in terms of complexities of prefixes, the notion of algorithmic Hausdorff dimension). The notion of $\KPt$-triviality is on the low end of this spectrum; here we consider sequences that have prefixes of minimal possible prefix complexity:

\begin{definition}
An infinite binary sequence $a_0a_1a_2\ldots$, is called \emph{$\KPt$-trivial} if its prefixes have minimal possible \textup(up to a constant\textup) prefix complexity, i.e., if 
      $$
\KP(a_0 a_1\ldots a_{n-1})=\KP(n)+O(1).      
      $$
\end{definition}

Note that $n$ can be reconstructed from $a_0\ldots a_{n-1}$, so $\KP(a_0\ldots a_{n-1})$ cannot be smaller than $\KP(n)-O(1)$. Note also that every computable sequence is $\KPt$-trivial, since $a_0\ldots a_{n-1}$ can be computed given~$n$. These two remarks together show that a $\KPt$-trivial sequence is very close to being computable. And indeed, if we were to replace prefix complexity by plain complexity~$\KS$ in the definition, the resulting notion, call it $\KS$-triviality, would be equivalent to being computable (it is in fact enough to have $\KS(a_0 a_1\ldots a_{n-1})\le\log n +O(1)$ to ensure that~$a_0a_1a_2\ldots$ is computable, see for example~\cite[problems 48 and 49]{usv}). Nonetheless, we shall see below that non-computable $\KP$-trivial sequences do exist. But before that, let us prove the following result, due to Chaitin.

\begin{theorem}\label{thm:trivial-halting-computable}
Every $\KPt$-trivial sequence is $\mathbf{0}'$-computable.
\end{theorem}

Here $\mathbf{0}'$ is the oracle for the halting problem.

\begin{proof}
Assume that the complexity of the $n$-bit prefix $(a)_n=a_0a_1\ldots a_{n-1}$ is $\KP(n)+O(1)$.
Recall that $(a)_n$ has the same information content as~ $(n,(a)_n)$, and use the formula for the complexity of a pair:
   $$
\KP((a)_n)=\KP(n,(a)_n)+O(1)=\KP(n)+\KP((a)_n\cnd n,\KP(n))+O(1);
   $$
This means that $$\KP((a)_n\cnd n,\KP(n))=O(1).$$ So $(a)_n$ belongs to a $\mathbf{0}'$-computable (given $n$) list of $n$-bit strings that has size~$O(1)$. Therefore, $a$ is a path in a $\mathbf{0}'$-computable tree of bounded width and is $\mathbf{0}'$-computable. Indeed, assume that the tree has $k$ infinite paths that all diverge before some level $N$. At levels after $N$ we can identify all the paths, since all other nodes have finite subtrees above them, and we may wait until only $k$ candidates remain. \qed
\end{proof}

The existence of non-computable $\KPt$-trivial sets is not obvious, but not very difficult to establish, even if we additionally require the set to be enumerable. Here we identify a set $A$ with its characteristic sequence $a_0a_1a_2\ldots$ (where $a_i=1$ if and only if $i\in A$).

\begin{theorem}\label{thm:existence}
There exists an enumerable undecidable $\KPt$-trivial set.
\end{theorem}

This result was proven by Solovay in the 1970s.

Let us make some preparations for this proof which will be useful in the rest of the paper. First, to deal with $\KP$-triviality, it is easier to use a priori discrete probability~$\m$ instead of $\KP$ (see, e.g., \cite{uppsala-notes} or \cite{usv} for more background on $\m$ and its relation to prefix complexity; recall that $\m(x)$ coincides with $2^{-\KPt(x)}$ up to a $\Theta(1)$-factor). In this setting, a sequence~$a_0a_1\ldots$ is $K$-trivial if and only if
\[
\m(a_0a_1\ldots a_{n-1}) \geq \m(n)/O(1)
\]

Since $\m$ is multiplicatively maximal  among all lower semicomputable semi\-mea\-sures, the statement of Theorem~\ref{thm:existence} can be rephrased as follows:
\begin{quote}
for every lower semicomputable semimeasure~$\mu$, there exist an enumerable set $A$ with its characteristic sequence $a=a_0a_1a_2\ldots $ and a lower semicomputable semimeasure~$\nu$ such that $$\nu(a_0a_1\ldots a_{n-1}) \geq \mu(n)/O(1).$$
\end{quote}
In fact, we need this only for $\mu=\m$, but the argument works for any lower semicomputable semimeasure $\mu$ (and in any case the statement for $\mu=\m$ is stronger and implies the same inequality for every $\mu$, though with a different constant in $O(1)$-notation).

\begin{proof}
To prove Theorem~\textup{\ref{thm:existence}}, let us assume that $\mu$ is a lower semicomputable semimeasure. We want to build an enumerable set~$A$ that corresponds to a sequence $a_0a_1a_2\ldots$, together with a lower semicomputable semimeasure~$\nu$ such that~$\nu(a_0a_1\ldots a_{n-1})$ matches $\mu(n)$ up to a multiplicative constant for all~$n$. When we see that~$\mu$ increases the weight of some $n$, we should respond by increasing the $\nu$-weight of some node (=string) of length~$n$, achieving the same weight (up to $O(1)$-factor). Moreover, all these nodes should lie on the tree path that corresponds to some enumerable set~$A$.  

Doing this would be trivial for a computable sequence $a$: constructing $\nu$, we just place at $a_0a_1\ldots a_{n-1}$ the same weight as the current value of~$\mu$ at $n$. This (evidently) gives a semimeasure since the sum of the weights is the same for $\mu$ and $\nu$.

But we want $A$ to be non-computable. To achieve this, we will ensure that $A$ is simple in Post's sense. Recall that a simple set is an enumerable set $A$ with infinite complement such that $A$ has non-empty intersection with every $W_n$ that is infinite. Here by $W_n$ we denote the $n$-th enumerable set in some natural numbering of all (computably) enumerable sets. As in Post's cllassical construction, we want for every $n$ to add some element of $W_n$ greater than $2n$ into $A$, and then forget about $W_n$. The bound $2n$ guarantees that $A$ has infinite complement. In Post's construction the elements are added without reservations: as soon as some element that is greater than $2n$ is discovered in $W_n$, it is added to~$A$. But now, when adding such an element to $A$, we have to pay something for this action. Indeed, when we add some number $u$ to $A$, the path in the Cantor space corresponding to it (i.e., $A$'s characteristic sequence) changes. The $\nu$-weights put on the node $a_0a_1\ldots a_{u-1}$ and all its extensions are lost, and should be recreated along the new path (starting from length $u$).   All this lost amount can be called the \emph{cost} of the action.

Now we can explain the construction. Initially our set $A$ is empty, and the corresponding path $a$ in Cantor space is all zeros. Observing the growth of the semimeasure $\mu$, we replicate the corresponding values along $a$. We also enumerate all $W_n$  in parallel. When a new element~$u$ is enumerated into $W_n$, we add this element to~$A$ if the following two conditions are satisfied:
\begin{itemize}
\item $u>2n$; 
\item the cost of adding $u$ is small, say, less than $2^{-n}$, so the total cost for all $n$ is bounded. 
\end{itemize}
Here the \emph{cost of the action} is the total $\nu$-weight we had placed on the nodes along the current path $a$ starting from level $u$: this weight is lost and needs to be replicated along the new path.
In this way the total $\nu$-weight is bounded. Indeed, the lost weight is bounded by $\sum_n 2^{-n}$ (recall that we take care of each $W_n$ at most once), and the replicated weight is bounded by $\sum \mu(n)$. 

If $W_n$ is infinite, it contains arbitrarily large elements, and the cost of adding $u$ is bounded by $$\mu(u)+\mu(u+1)+\mu(u+2)+\ldots,$$ which is guaranteed to go below the $2^{-n}$ threshold for large $u$. So for every infinite $W_n$, some element~$u$ of~$W_n$ will be added to~$A$ at some stage of the construction. As we have seen, the total $\nu$-weight is bounded by $\sum_n \mu(n) + \sum_n 2^{-n}$. By construction, we have $\nu(a_0a_1\ldots a_{n-1}) \geq \mu(n)$ for all~$n$. It remains to divide $\nu$ by some constant to make the total weight bounded by $1$.
     \qed
\end{proof}

This proof can be represented in a game form. In such a simple case this looks like an overkill, but the same technique is useful in more complicated cases, so it is instructive to look at this version of the proof. The game field consists of the set of the natural numbers (\emph{lengths}), the full binary tree, and sets $W_1,W_2,\ldots$ (of natural numbers). The opponent increases the weights assigned to lengths: each length has some weight that is initially zero and can be increased by the opponent at any moment by any non-negative rational number; the only restriction is that the total weight of all lengths should not exceed~$1$. Also the opponent may add new elements to any of the sets $W_i$; initially they are empty. We construct a path $a$ in the binary tree that is a characteristic sequence of some set $A$, initially empty, by adding elements to $A$; we also increase the weights of nodes of the binary tree in the same way as the opponent does for lengths; our total weight should not exceed~$2$.

One should also specify when the players can make moves. It is not important, since the rules of the game always allow each player to postpone moves. Let us agree that the players make their moves in turns and every move is finite: finitely many weights of lengths and nodes are increased by some rational numbers, and finitely many new elements are added to $W_i$ and $A$. This is the game with full information, the moves of one player are visible to the other one.

The game is infinite, and the winner is determined in the limit, assuming that both players obey the weight restrictions. Namely, we win if
\begin{itemize}
\item for the limit path $a$ our weight of $(a)_n$ is not less than the opponent's weight of $n$;
\item for each $n$, if $W_n$ is infinite, then $W_n$ has a common element with $A$.
\end{itemize}

The winning strategy is as described: we match the opponent's weight along the current path, and also we add some $u$ to $A$ and change the path, matching the opponent's weights along the new path, if $u$ belongs to $W_n$, is greater than $2n$ and the cost of the action, i.e.,  our total weight along the current path above $u$, does not exceed $2^{-n}$. 

This is a computable winning strategy. Indeed, the limit weights of all lengths form a converging series, so if $W_n$ is infinite, it has some element that is greater than $2n$ and for which the loss, bounded by the tail of this series, is less than $2^{-n}$.

Imagine now that we use this computable winning strategy against the ``blind'' computable opponent that ignores our moves and just enumerates from below the a priori probability (as lengths' weights) and the sets $W_i$ (the list contains all enumerable sets). Then the game is computable, our limit $A$ is an enumerable simple set, and our weights for the prefixes of $a$ (and therefore $\m((a)_n)$, since the limit weights form a lower semicomputable semimeasure) match $\m(n)$ up to $O(1)$-factor.

\section{$\KPt$-trivial and $\KPt$-low Sequences}

Now we know that non-computable $\KPt$-trivial sequences do exist.  Our next big goal is to show that they are computationally weak. Namely, they are \emph{$\KPt$-low} in the sense of the following definition.

Consider a bit sequence $a$; one can relativize the definition of prefix complexity using $a$ as an oracle (i.e., the decompressor algorithm used in the definition of complexity may use the values of $a_i$ in its computation). For every oracle this relativized complexity $\KP^a$ does not exceed (up to an $O(1)$ additive term) the non-relativized prefix complexity, since the decompressor may ignore the oracle. But it can be smaller or not, depending on~$a$.

\begin{definition}
A sequence $a$ is $\KPt$-low if $\KP^a(x)=\KP(x)+O(1)$.
\end{definition}

In other words, $\KPt$-low oracles are useless for compression (or, more precisely, decompression) purposes. 

Obviously, computable oracles are low; the question is whether there exist non-computable low oracles.  Note that \emph{``classical'' undecidable sets, like the halting problem, are not $K$-low}: with oracle $\mathbf{0}'$ the table of complexities of all $n$-bit strings has complexity $O(\log n)$, but its non-relativized complexity is $n-O(1)$.  One can also consider the relativized and non-relativized complexities of the prefixes of Chaitin's $\Omega$-numbers: the $n$-bit prefix has complexity about $n$ but its $\mathbf{0}'$-relativized complexity is about $\log n$, since $\Omega$-numbers are $\mathbf{0}'$-computable. Note also that \emph{$\KPt$-low oracles are $\KPt$-trivial}, since $\KP^a ((a)_n)=\KP^a(n)+O(1)$: the sequence $a$ is computable in the presence of oracle $a$. 

It turns out that the reverse implication is true, and all $\KPt$-trivial sequences are $\KPt$-low. This is quite surprising. For example, one may note that \emph{the notion of a $\KPt$-low sequence is Turing-invariant}, i.e., depends only on the computational power of the sequence, but for $\KPt$-triviality there are no reasons to expect this, since the definition deals with prefixes. 

On the other hand, it is easy to see that \emph{if $a$ and $b$ are two $\KPt$-trivial sequences, then their join} (the sequence $a_0b_0a_1b_1a_2b_2\ldots$) \emph{is also $\KPt$-trivial}. Indeed, as we have mentioned, the $\KPt$-triviality of a sequence $a$ means that $\KP((a)_n\cnd n, \KP(n))=O(1)$. If at the same time $\KP((b)_n\cnd n, \KP(n))=O(1)$, then the na\"\i ve bound for the complexity of a pair guarantees that $$\KP((a)_n,(b)_n\cnd n, \KP(n))=O(1),$$ so $$\KP(a_0b_0a_1b_1\ldots a_{n-1}b_{n-1}\cnd n,\KP(n))=O(1),$$ and therefore $$\KP(a_0b_0a_1b_1\ldots a_{n-1}b_{n-1})=\KP(n)+O(1).$$ It remains to note that $\KP(n)=\KP(2n)+O(1)$ and that we can extend the equality to sequences of odd length, since adding one bit changes the complexity of the sequence and its length only by $O(1)$. The analogue result for $\KPt$-low sequences is not obvious: if each of the sequences $a$ and $b$ separately do not change the complexity function, why should their join be equally powerless in that regard? The usual proof of this result uses the equivalence between triviality and lowness.

The proof of the equivalence (every $\KPt$-trivial sequence is $\KPt$-low) requires a rather complicated combinatorial construction. It may be easier to start with a weaker statement:  \emph{no $\KPt$-trivial sequence} (used as an oracle) \emph{computes the halting problem}. This statement is indeed a corollary of the equivalence result, since $\mathbf{0}'$ (the halting set) is not $\KPt$-low, as we have seen, and every sequence computable with a $\KPt$-low oracle is obviously $\KPt$-low. The proof of this weaker statement is given in the next section. On the other hand, the full proof (hopefully) can be understood without the training offered in the next section, so the reader may also skip it and go directly to Section~\ref{sec:trivial-low}.

\section{$\KPt$-trivial Sequences Cannot Compute $\mathbf{0}'$}

In this section we prove that a $\KPt$-trivial sequence cannot compute $\mathbf{0}'$, in the following equivalent version:

\begin{theorem}\label{thm:incompleteness}
No $\KPt$-trivial sequence can compute all enumerable sets.
\end{theorem}

Note that, together with the existence result proved above (Theorem~\ref{thm:existence}),  this theorem provides an answer to classical \emph{Post's problem}, the question whether non-complete enumerable undecidable sets exist.
\smallskip

The rest of the section is devoted to the proof of Theorem~\ref{thm:incompleteness}.

\subsubsection*{The Game Template.}

To make the proof of this theorem more intuitive, we first reformulate it in terms of a two-player game. Imagine that we want to prove that all $\KPt$-trivial sequences have some property $P$, and our opponent wants to show that we are wrong, i.e., to construct a $\KPt$-trivial sequence $a$ that does not have the property $P$. We already know that all $\KPt$-trivials are $\mathbf{0}'$-computable, so we may assume that our opponent presents a $\mathbf{0}'$-computable sequence~$a$ as a computable pointwise approximation (using Shoenfield's limit lemma). At every moment of the game the opponent chooses some values $a_i$; they may be changed during the game, but for each $i$ the number of changes in $a_i$ during the game should be finite, otherwise the opponent loses.

We want to show that the sequence constructed by the opponent is either not $\KPt$-trivial or has property $P$. For that we challenge the opponent by building a semimeasure~$\mu$ on integers by gradually increasing the weights of each integer. Recall the proof of Theorem~\ref{thm:existence}; now the opponent tries to certify that $a$ is $\KPt$-trivial and is therefore in the same position in which we were in that proof. In other words, he is obliged to match our increases along the path $a$, i.e., he must construct a semimeasure~$\nu$ on strings such that $\nu((a)_n) \geq \mu(n)/O(1)$ for the limit sequence $a$. At the same time the opponent needs to ensure that the limit sequence does not have the property $P$. In other terms, the opponent wins if (1)~the limit sequence exists and does not have the property $P$; (2)~$\nu$ is a semimeasure (the sum of all weight increases is bounded by $1$); (3)~$\nu((a)_n) \geq \mu(n)/O(1)$.

If we have a computable winning strategy in this game, then every $\KPt$-trivial sequence $a$ has property $P$. Indeed, assume that there exists some $\KPt$-trivial sequence that does not have this property. Then, by Theorem~\ref{thm:trivial-halting-computable}, the opponent can present this sequence as a computable pointwise approximation, and also use increasing approximations to $\m$ (on strings) for $\nu$. Our computable strategy will then generate some lower semicomputable semimeasure $\mu$, and the inequality $\nu((a)_n) \geq \mu(n)/O(1)$ is guaranteed by the maximality of $\m$ and the triviality of $a$, so the opponent wins against our winning strategy---a contradiction. 

\begin{remark}
One may also note that if the opponent has a computable winning strategy in the game, then there exists a $\KPt$-trivial sequence $a$ that does not have the property $P$. Indeed, let this strategy play and win against $\m$ (as $\mu$); the resulting sequence will be $\KPt$-trivial and will not have the property $P$. So the question whether all $\KPt$-trivial sequences have property $P$ or not can be resolved by providing a computable winning strategy for one of the players. 
\end{remark}

\subsubsection*{The Game for Theorem~\ref{thm:incompleteness}.}

We follow this scheme (in slightly modified form) and consider the following game. The opponent approximates some sequence~$a$ by changing the values of Boolean variables $a_0,a_1,a_2,\ldots$, so $a(i)$ is the limit value of $a_i$. (For the $i$th bit of $a$ we use the notation $a(i)$ instead of usual $a_i$, since $a_i$ is used as the name of $i$th variable.) He also assigns increasing weights to strings; the total weight should not exceed $1$. We assume that initially all weights are zeros. We also assume that the initial values of the $a_i$'s are zeros (just to be specific).

We assign increasing weights to integers (\emph{lengths}); the sum of our weights is also bounded by $1$. Since the property $P$ says that $a$ computes all enumerable sets, we also challenge this property and construct some set $W$ by irreversibly adding elements to it.
\smallskip

The opponent wins the game if
\begin{itemize}
\item each variable $a_i$ is changed only finitely many times (so some limit sequence $a$ appears);
\item the (opponent's) limit weight of $(a)_i$, the $i$-bit prefix of $a$, is greater than our limit weight of $i$, up to some multiplicative constant;
\item the set $W$ is Turing-reducible to $a$.
\end{itemize}

Again, it is enough to show that we can win this game using a  computable winning strategy. Indeed, assume that some $\KPt$-trivial $a$ computes $\mathbf{0}'$. We know that $a$ is limit computable, so the opponent can computably approximate it, and at the same time approximate from below the a priori probabilities $\m(s)$ for all strings $s$ (ignoring our moves). Our strategy will then behave computably, generating some lower semicomputable semimeasure on lengths,  and some enumerable set~$W$. Then, according to the definition of the game, either this semimeasure is not matched by $\m((a)_i)$, or $W$ is not Turing-reducible to $a$. In the first case $a$ is not $\KPt$-trivial; in the second case $a$ is not Turing-complete.

\subsubsection*{Reduction to a Game with Fixed Machine and Constant.}
How can we computably win this game? First we consider a simpler game where the opponent has to declare in advance some constant $c$ that relates the semimeasures constructed by the two players, and the machine $\Gamma$ that reduces $W$ to $a$. Imagine that we can win this game: assume that for each $c$ and $\Gamma$ we have a uniformly computable strategy that wins in this $c$-$\Gamma$-game, defined in a natural way. Since the constant $c$ in the definition of the $c$-$\Gamma$-game is arbitrary, we may use $c^2$ instead of $c$ and assume by scaling that we can force the opponent to spend more than $1$ while using only $1/c$ total weight and allowing him to match our moves up to factor $c$.

Now we mix the strategies for different $c$ and $\Gamma$ into one strategy. Note that two strategies that simultaneously increase weights of some lengths can only help each other, so we only need to ensure that the sum of the increases made by all strategies is bounded by~$1$.  More care is needed for the other condition related to the set $W$. Each of the strategies constructs its own $W$, so we should isolate them. For example, to mix two strategies, we split $\mathbb{N}$ into two parts $N_1$ and $N_2$, say, odd and even numbers, and let the first and second strategy construct a subset of $N_1$ and $N_2$, respectively. Of course, then each strategy is not required to beat the machine $\Gamma$; it should beat its restriction to $N_i$ (the composition of $\Gamma$ and the embedding of $N_i$ into~$\mathbb{N}$). In a similar way we can mix countably many strategies (splitting $\mathbb{N}$ into countably many infinite sets in a computable way). 

It remains to consider some computable sequence $c_i>0$ such that ${\sum 1/c_i \le 1}$ and a computable sequence $\Gamma_i$ that includes every machine $\Gamma$ infinitely many times. (The latter is needed because we want every $\Gamma$ to be beaten with arbitrarily large constant $c$.) Combining the strategies for these games as described, we get a computable winning strategy for the full game.

When constructing a wnning strategy for the $c$-$\Gamma$-game, it is convenient to scale this game and require the opponent to match our weights exactly (without any factor) but allow him to use total weight $c$ instead of~$1$.  We will prove the existence of the winning strategy by induction: assuming that a strategy for some $c$ is given,  we construct a strategy for a bigger~$c'$. Let us first construct the strategy for $c<2$. 

\subsubsection*{Winning a Game with $c<2$: Strong Strings.}

This winning strategy deals with some fixed machine $\Gamma$ and ensures $\Gamma^a\ne W$ at one fixed point, say, $0$ (i.e., the strategy ensures that $0\in\Gamma^a\not\Leftrightarrow0\in A$); the other points are not used. Informally, we wait until the opponent puts a lot of weight on strings that imply $0\notin \Gamma^a$. If this never happens, we win in one way; if it happens, we then add $0$ to $W$ and win in a different way.

Let us explain this more formally. We say that a string $u$ is \emph{strong} if it (as a prefix of $a$) enforces that $\Gamma^a(0)$ is equal to $0$, i.e., $\Gamma$ outputs $0$ on input $0$ using only oracle answers in $u$. Simulating the behavior of $\Gamma$ for different oracles, we can enumerate all strong strings. During the game we look at the following quantity:
\begin{quote}
\emph{the total weight that our opponent has put on all known strong strings}.
\end{quote}
This quantity may increase because the opponent distributes more weight or because we discover new strong strings, but it never decreases. We try to force the opponent to increase this quantity (see below how). As we shall see, if he refuses, he loses the game, and the element $0$ remains outside $W$. If the quantity comes close to $1$, we change our mind and add $1$ into $W$. After that all the weight put on strong strings is lost for the opponent: they cannot be the prefixes of $a$ such that $\Gamma^a=W$, if $1\in W$. So we can make our total weight equal to $1$ in an arbitrary way (adding weight somewhere if the total weight was smaller than~$1$), and to counter this the opponent needs to use additional weight $1$ along some final $a$ that avoids all strong strings, therefore his weight comes close to $2$.

\subsubsection*{Winning a Game with $c<2$: Gradual Increase.}

So our goal is to force the opponent to increase the total weight of strong strings (or nodes, since we identify strings with nodes in the full binary tree).  Let us describe our strategy as a set of substrategies (processes) that run in parallel. For each strong node $x$ there is a process $P_x$; we start it when we discover that $x$ is strong. This process tries to force the opponent to increase the weight of some extension of $x$, i.e., some node above $x$ (this node is automatically strong); we want to make the lengths of these strings different, so for every string $x$ we fix some number $l_x$ that is greater than $|x|$, the length of $x$.

The process $P_x$ is activated when the current path (i.e., the characteristic function of the current approximation to~$a$) goes through node $x$. Otherwise $P_x$ sleeps; it may happen that some $P_x$ never becomes active. When awake, $P_x$ always sees that the current path goes through $x$. The process $P_x$ gradually increases the weight of length $l_x$: it adds some small $\delta_x$ to the weight of $l_x$ and waits until the opponent matches\footnote{A technical remark: note that in our description of the game we have required that the opponent's weight along the path is strictly greater than our weight: if this is true in the limit, it happens at some finite stage.} this weight along the current path (whatever this path is), then increases the weight again by $\delta_x$, etc. The value of $\delta_x$ is fixed for each $x$ in such a way that $\sum_x \delta_x$ is small (formally: we need it to be smaller than~$1-c/2$).  The process repeats this increase by $\delta_x$ until it gets a termination signal from the supervisor (see below for the conditions when this happens). Note that at any moment the current path may change in such a way that~$x$ is not an initial segment of it anymore. Then $P_x$ is suspended and wakes up only when, due to a later change of the current path, $x$ lies on it again (and this may never happen).

The supervisor sends the termination signal to all the processes when (and if) the total weight they have distributed goes above a fixed threshold close to~$1$ (formally, we need this threshold to be greater than $c/2$). After that the strategy adds $0$ to $W$, as explained above. 

Let us show that this is indeed a winning strategy. Consider a game where it is used. By construction, we do not violate the weight restriction. If the opponent has no limit path, he loses, so we can assume that some limit path~$a$ exists. There are two possible cases:

\begin{itemize}
\item Case 1: The processes $P_x$ never reach the threshold for the total weight used, and no termination signal is ever sent (thus $0$ never enters~$W$). This can be because of two reasons:
	\begin{itemize}
	\item There is no strong node on the limit path~$a$. In this case, the opponent loses because $\Gamma^a(0) \ne 0$ while $0 \notin W$, so $\Gamma^a \ne W$. 
	\item There exists a strong node $x$ on the limit path~$a$, and its associated process $P_x$ remains active from some point onward, but the opponent refuses to match our weight at length $l_x$. In this case the opponent fails to match our weight on some prefix of~$a$, and thus loses. 
	\end{itemize}
\item Case 2: The processes $P_x$ do reach the fixed total threshold, the termination signal is sent, and $0$ is added to~$W$. As we observed, all the weight we put on lengths was matched by our opponent on strong nodes, except for some small amount (at most $\sum_x \delta_x$). Now all the opponent's weight is lost since he must change the path and the limit path does not have strong prefixes. Then we distribute the remaining weight arbitrarily and win.
\end{itemize}

This finishes the explanation on how to win the $c$-$\Gamma$-game for $c<2$.

\subsubsection*{Induction Statement.}

The idea of the induction step is simple: instead of forcing the weight increase for some extension of a strong node $u$ directly, we recursively call the described strategy at the subtree rooted at $u$, adding or not adding some other element to $W$ instead of $0$. This cuts our costs almost in half, since we know how to win the game with $c$ close to~$2$. In this way we can win the game for arbitrary $c<3$, and so on.

To be more formal, we consider a recursively defined process  
$P(k,x,\alpha, L, M)$ with the following parameters:
\begin{itemize}
\item $k>0$ is a rational number, the required coefficient of weight increase;
\item $x$ is the root of the subtree where the process operates;
\item $\alpha>0$ is also a rational number, our ``budget'' (how much weight we are allowed to use);
\item $L$ is an infinite set of integers (lengths where our process may increase weight);\footnote{To use infinite sets as parameters, we should restrict ourselves to some class of infinite sets. For example, we may consider infinite decidable sets and represent them by programs enumerating their elements in increasing order.}
\item $M$ is an infinite set of integers (numbers that our process is allowed to add to~$W$).
\end{itemize}

The process can be started or resumed only when $x$ is a prefix of the current path $a$, and is suspended when $a$ changes and this is no more true. The process then sleeps until $x$ is an initial segment of the current path again. It is guaranteed that $P$ never violates the rules (about $\alpha$, $L$, and $M$). Running in parallel with other processes (as part of the game strategy) and assuming that other processes do not touch lengths in $L$ and numbers in $M$, the process $P(k,x,\alpha,L,M)$ guarantees, if not suspended forever or terminated externally, that one of the following possibilities is realized:
\begin{itemize}
\item the limit path~$a$ does not exist;
\item $W\ne \Gamma^a$ for limit $a$;
\item the opponent never matches some weight put on some length in $L$;
\item the opponent spends more than $k\alpha$ weight on nodes above $x$ with lengths in~$L$.
\end{itemize}
    
\subsubsection*{Base Case: $k<2$.}     
     
Now we can adapt the construction of the previous section and construct a process $P(k,x,\alpha,L,M)$ with these properties for arbitrary $k<2$. For each $y$ above $x$ we select some $l_y\in L$ greater than the length of $y$, different for different $y$, and also select some positive $\delta_y$ such that $\sum\delta_y$ is small compared to the budget $\alpha$. We choose some $m\in M$ and consider $y$ (a node above $x$) as \emph{strong} if it guarantees that $\Gamma^a(m)=0$. Then for all strong $y$ we start the process $P_y$ that  is activated when $y$ is in the current path and increases the weight of $l_y$ in $\delta_y$-steps waiting until the opponent matches it. We terminate all the processes when (and if) the total weight used becomes close to $\alpha$, and then add $m$ to $W$, thus rendering useless all the weight placed by the opponent on strong nodes.

The restrictions are satisfied by the construction. Let us check that the declared goals are achieved. If there is no limit path, there is nothing to check. If the limit path $a$ does not go through $x$, we have no obligations (the process is suspended forever). So we assume that the limit path goes through $x$. 

Assume first that the total weight used by all $P_y$ did not come close to $\alpha$, so the termination signal was not sent. In this case $m\notin W$.  If there is no strong node on the limit path $a$, then $\Gamma^a(m)$ is not $0$, so $W\ne\Gamma^a$. If there is a strong node $y$ on the limit path, then the process $P_y$ was started and worked without interruptions, starting from some moment. So either some of the $\delta_y$-increases was not matched (third possibility) or the termination signal was sent (so there are no obligations). 

It remains to consider the case when $m$ was added to $W$ and termination signal was sent to all $P_y$.  In this case the total weight used is close to $\alpha$, and after adding $m$ to $W$ it is lost, so either our weight is not matched or almost $2\alpha$ is spent by the opponent on nodes above $x$ (recall that all processes are active only when the current path goes through~$x$). 

\subsubsection*{Induction Step.}

The induction step is similar: we construct the process $$P(k,x,\alpha,L,M)$$ in the same way as for the induction base. The difference is that instead of $\delta_y$-increasing the weight of $l_y$ the process $P_y$ now recursively calls $$P(k',y,\delta_y\,,L',M')$$  with some smaller $k'$, say, $k'=k-0.5$, the budget $\delta_y$, and some $L'\subset L$ and $M'\subset M$. If the started process forces the opponent to spend more than $k'\delta_y$ on the nodes above $y$ with lengths in $L'$ and terminates, then a new process $$P(k',y,\delta_y\,,L'',M'')$$ is started for some other $L''\subset L$ and $M''\subset M$, etc. All the subsets $L',L'',\ldots$ should be disjoint, and also disjoint for different $y$, as well as $M',M'',\ldots$. So we should first of all split $L$ into a sum of disjoint infinite subsets $L_y$ parametrized by $y$ and then split each $L_y$ into $L_y'+L_y''+\ldots$ (for the first, second, etc. recursive calls). The same is done for $M$, but here, in addition to the sets $M_y$, we select some~$m$ outside all $M_y$. We add this~$m$ to $A$ when our budget is exhausted (thus forcing the opponent to spend more weight). As before, strong nodes are defined as those that guarantee $\Gamma^a(m)=0$. 

We start the processes $P_y$ as described above: each of them makes a potentially infinite sequence of recursive calls with the same $k'=k-0.5$ and budget $\delta_y$. The process $P_y$ is created for each discovered strong node $y$, but is sleeping while $y$ is not on the current path. We take note of the total weight used by all $P_y$ (for all $y$) and send a termination signal to all $P_y$ when this weight comes close to the threshold $\alpha$, so it never crosses this threshold.

Let us show that we achieve the declared goal, assuming that the recursive calls fulfill their obligations.  First, the restrictions about $L$, $M$ and $\alpha$ are guaranteed by the construction. If there is no limit path, we have no other obligations. If the limit path exists but does not go through $x$, our process will  be suspended externally, and again we have no obligations. So we may assume that the limit path goes through $x$, and that our process is not  terminated externally. If the weight used by all $P_y$ did not cross the threshold, and the limit path does not go through any strong node (defined using $m$), then $W\ne \Gamma^a$ for the limit path $A$, since $m\notin W$ and $\Gamma^a(m)$ does not output $0$. If the limit path goes through some strong $y$, the process $P_y$ will be active starting from some point onward, and makes recursive calls $P(k',y,\delta_y\,,L',M')$, $P(k',y,\delta_y\,,L'',M'')$, etc. Now we use the inductive assumption and assume that these calls achieve their declared goals. Consider the first call. If it succeeds by achieving one of three first alternatives (among the four alternatives listed above), then we are done. If it succeeds by achieving the fourth alternative, i.e., by forcing the opponent to spend more than $k'\delta_y$ on the weights from $L'$, then the second call is made, and again either we are done or the opponent spends more than $k'\delta_y$ on the weights from $L''$. And so on: at some point we either succeed globally, or exhaust the budget and our main process sends the termination signal to all $P_y$.  So it remains to consider the latter case. Then all the weight spent, except for the $\delta_y$'s for the last call at each node, is matched by the opponent with factor $k'$, and on the final path the opponent has to match it with factor $1$, so we are done (assuming that $k<k'+1$ and $\sum_y \delta_y$ is small enough).

This finishes the induction step, so we can win every $c$-$\Gamma$-game by calling the recursive process at the root. As we have explained, this implies that  $\KPt$-trivial sets do not compute $\mathbf{0}'$.

\section{$\KPt$-trivial Sequences Are $\KPt$-low}\label{sec:trivial-low}

Now we want to prove the promised stronger result~\cite{nies-main}:

\begin{theorem}
All $\KPt$-trivial sequences are $\KPt$-low.
\end{theorem}

\begin{proof}
In this theorem the property $P$ that we want to establish for an arbitrary $\KPt$-trivial sequence $a$ says that $\KP^a(x)\ge \KP(x)-O(1)$ for all $x\in\mathbb{B}^*$, or that $$\m^a(x)\le \m(x)\cdot O(1) \text{ for all $x\in\mathbb{B}^*$}.$$ Let us represent $\m^a(\cdot)$ in the following convenient way. The sequence $a$ is a path in a full binary tree. Imagine that at every node of the tree there is a label of the form $(i,\eta)$ where $i$ is an integer, and $\eta$ is a non-negative rational number. This label is read as ``please add $\eta$ to the weight of $i$''. We assume that the labelling is computable. We also require that for every path in the tree the sum of all rational numbers along the path does not exceed $1$. Having such a labelling, and a path $a$, we can obey all the labels along the path, and obtain a semimeasure on integers. This semimeasure is semicomputable with oracle $a$.

This construction is general in the following sense. Consider a machine $M$ that generates a lower semicomputable discrete semimeasure $\m^a$ when given access to an oracle $a$.  We can find a computable labelling that gives the same semimeasure  $\m^a$ (in the way described) for every oracle $a$. (Note that, according to our claim, the labelling does \emph{not} depend on the oracle.)  Indeed, we may simulate the behavior of $M$ for different oracles $a$, and look at the part of $a$ that has been read when some increase in the output semimeasure happens. This can be used to create a label $(i,\eta)$ at some tree node $u$: the number $i$ is where the increase happened, $\eta$ is the size of the increase, and $u$ is the node that guarantees all the oracle answers used before the increase happened. We need to make the labelling computable; also, according to our assumption, each node has only one label (adds weight only to one object). Both requirements can be easily fulfilled by postponing the weight increase: we push the queue of postponed requests up the tree. If the sum of the increase requests along some path $a$ becomes greater than $1$, this means that for this path $a$ we do not obtain a semimeasure. As usual, we can trim the requests and guarantee that we obtain semimeasures along all paths, without changing the existing valid semimeasures. 

We may assume now that some computable labelling is fixed that corresponds to the universal machine: for every path $a$ the semimeasure resulting from fulfilling all requests along $a$, equals $\m^a$.

\subsubsection*{Game Description.}

As in the previous section, we prove the theorem by showing the existence of a winning strategy in some game, which follows the same template. 

As before, the opponent approximates some sequence~$a$ by changing the values of Boolean variables $a_0,a_1,a_2,\ldots$ and assigns increasing weights to strings; the total weight should not exceed~$1$ (we again assume that initially all weights and the values of the $a_i$'s are zeros) while we assign increasing weights to integers (\emph{lengths}); the sum of our weights is also bounded by $1$. 
\smallskip

Moreover (this is the part of the game tailored for the theorem to be proven), throughout the game we also assign increasing weights to another type of integers, called \emph{objects}: on these we compare our semimeasure with the semimeasure $\m^a$ determined by the opponent's limit path $a$.\footnote{Formally speaking, we construct two semimeasures on integers; to avoid confusion, it is convenient to call their arguments ``lengths'' and ``objects''.}

\medskip

The opponent wins the game if all of the following three conditions are satisfied:
\begin{itemize}
\item the limit sequence $a$ exists;
\item the opponent's semimeasure along the path exceeds our semimeasure on lengths up to some constant factor, i.e., there exists some $c>0$ such that for all~$i$ the opponent's weight of the prefix $(a)_i$ is greater than our weight of $i$ divided by $c$; for brevity we say in this case that the opponent's semimeasure $*$-\emph{exceeds} our semimeasure.
\item our semimeasure on objects does not \hbox{$*$-exceed} $\m^a$.
\end{itemize}

Once again it is enough to construct a computable winning strategy in this game. Also, as in the previous section, we can consider an easier (for us) version of the game where the opponent starts the game by declaring some constant~$c$ that he plans to achieve for the second condition, and we need to beat only this~$c$. If we can win this game for any~$c=2^{2k}$ declared in advance, then by scaling we can win the $2^k$-game using only $2^{-k}$ of our capital, and it then suffices to combine all the corresponding strategies (we also assume that the total weight on objects for the $k$th strategy is bounded by $2^{-k}$, but this is for free, since we only need to $*$-exceed $\m^a$ without any restrictions on the constant). So it remains to win the game for each~$c$. And again, it is convenient to scale that game and assume that the opponent needs to match our weights on lengths exactly (not up to $1/c$-factor) while his total weight is bounded by $c$ (not~$1$).

\subsubsection*{Winning the Game for $c<2$.}

For $c=1$ the game is trivial, since we require that the opponent's weight along the path is strictly greater than our weight on lengths, so it is enough to assign weight~$1$ to some length. We start our proof by explaining the strategy for the case $c<2$.

The idea can be explained as follows.  The na\"\i ve strategy is to assume all the time that the current path~$a$ is final, and to just assign the weights to objects according to $\m^a$, computed based on the current path $a$. (In fact, at each moment we look at some finite prefix of $a$ and follow the labels that appear on this prefix.)  If indeed $a$ never changes, this is a valid strategy: we achieve $\m^a$, and never exceed the total weight $1$ due to the assumption about the labels. But if the path suddenly changes, then all the weight placed because of nodes on the old path which are now outside the new path, is lost. If we now follow all the labels on the new path, then our total weight  on objects may exceed $1$ (the total weight was bounded only along every path individually, but now we have placed weight according to labels both on the old path and on the new path).

There is some partial remedy: we may match the weights only up to some constant, say, use only $1\%$ of what the labels ask. This is possible since the game allows us to match the measure with arbitrary constant factor. This way we can tolerate up to $100$ changes in the path (each new path generates new weight of at most $0.01$). However, this does not really help since the number of changes is (of course) unbounded. In fact, the strategy described so far \emph{must} fail, as otherwise it would prove that all $\mathbf{0}'$-computable sets are $\KPt$-low, which is certainly not the case. For a successfull proof we must take advantage of the fact that $a$ is $\KPt$-trivial.

How can we discourage the opponent from changing the path? Like in the previous proof we may assign a non-zero weight to some length and wait until the opponent matches this weight along the current path. This provides an incentive for the opponent not to leave a node where he has already put weight: if he does, this weight would be wasted, and he would be forced to put the same weight along the final path a second time. After that we may act as if the final path goes through this node and follow the labels (as described). Doing this, we know at least that if later the path changes and we lose some weight, the opponent loses some weight, too. This helps if we are careful enough.

Let us explain the details. It would be convenient to represent the strategy as a set of parallel processes: for each node~$x$ we have a process~$P_x$ that is awake when $x$ is a prefix of the current path, and sleeps when $x$ is not. When awake, the process $P_x$ tries to create the incentive for the opponent not to leave~$x$, by forcing him to increase the weight of some node above~$x$. To make the processes more independent and to simplify the analysis, let us assume that for every node $x$ some length $l_x\ge |x|$ is chosen, lengths assigned to different nodes are different, and $P_x$ increases only the weight of $l_x$.

Now we are ready to describe the process~$P_x$. Assume that node $x$ has label $(i,\eta)$ that asks to add $\eta$ to the weight of object $i$. The process $P_x$ increases the weight of $l_x$, adding small portions to it and waiting after each portion until the opponent matches this increase along the current path (i.e., in the $l_x$-bit prefix of the current path). If and when the weight of $l_x$ reaches $\varepsilon\eta$ (where $\varepsilon$ is some small positive constant; the choice of $\varepsilon$ depends on $c$, see below), the process increases the weight of object $i$ by $\varepsilon\eta$ as well and terminates. 

The processes $P_x$ for different nodes $x$ run in parallel independently, except for one thing: just before the total weight spent by all processes together would exceed $1$, we terminate them, blocking the final weight increase that would have brought the total weight above $1$. After that our strategy stops working and hopes that the opponent would be unable to match already existing weights not crossing the threshold $c$.

Concerning the small portions of weight increases mentioned above: for each node $x$ we choose in advance the size $\delta_x$ of the portions used by $P_x$, in such a way that $\sum_x \delta_x <\varepsilon$. Note that here we use the same small $\varepsilon$ as above. In this way we guarantee that the total loss (caused by last portions that were not matched because the opponent changes the path instead and does not return, so the process is not resumed) is bounded by $\varepsilon$.

It remains to prove that this strategy wins the $c$-game for~$c$ close to $2$, assuming that $\varepsilon$ is small enough. First note two properties that are true by construction:
\begin{itemize}
\item the sum of our weights for all lengths does not exceed $1$;
\item at every moment the sum of (our) weights for all objects does not exceed the sum of (our) weights for all lengths.
\end{itemize}
Indeed, we stop the strategy just before violating the first requirement, and the second is guaranteed for each $x$-process and therefore for the entire strategy.

If there is no limit path, the strategy wins the game by definition. So assume that a limit path~$a$ exists. Now we count separately the weights used by processes $P_x$ for $x$'s on the limit path $a$, and for others (incomparable with $a$). Since the weights for $x$ are limited by $\varepsilon\cdot($the request in $x)$, and the sum of all requests along $a$ is at most $1$, the sum of the weights along $a$ is bounded by $\varepsilon$. Now there are two possibilities: either the strategy was stopped when trying to cross the threshold, or it runs indefinitely. 

In the first case the total weight is close to $1$: it is at least $1-\varepsilon$, since the next increase will cross $1$, and all the portions $\delta_x$ are less than $\varepsilon$. So the weight used by processes outside $a$ is at least $1-2\varepsilon$, and if we do not count the last (unmatched) portions, we get at least $1-3\varepsilon$ of weight that the opponent needs to match twice: it was matched above $x$ for $P_x$, and then should be matched again along the limit path (that does not go through $x$; recall that we consider the nodes outside the limit path). So the opponent needs to spend at least $2-6\varepsilon$, otherwise he loses.

In the second case each process $P_x$ for $x$ on the limit path is awake starting from some point onward, and is never stopped, so it reaches its target value $\varepsilon\eta$ and adds $\varepsilon\eta$ to the object $i$, if $(i,\varepsilon)$ is the request in node $x$. So our weights on the objects match $\m^a$ for limit path $a$ up to factor $\varepsilon$, and the opponent loses. We know also that the total weight on objects does not exceed $1$, since it is bounded by the total weight on lengths.

We therefore have constructed a winning strategy for the $2-6\varepsilon$ game, and by choosing a small $\varepsilon$ we can win the $c$-game for any given $c<2$.

\subsubsection*{Using This Strategy on a Subtree.}

To prepare ourselves for the induction, let us look at the strategy previously described and modify it for use inside a subtree rooted at some node $x$. We also scale the game and assume that we have some budget $\alpha$ that we are allowed to use (instead of total weight $1$). To guarantee that the strategy does not interfere with other actions outside the subtree rooted at~$x$, we agree that it uses lengths only from some infinite set $L$ of length and nobody else touches these lengths. Then we can assign $l_y\in L$ for every $y$ in the subtree and use them as before. 

Let us describe the strategy in more details. It is composed of processes $P_y$ for all $y$ above $x$. When $x$ is not on the actual path, all these processes sleep, and the strategy is sleeping. But when the path goes through~$x$, some processes~$P_y$ (for $y$ on the path) become active and start increasing the weight of length~$l_y$ by small portions $\delta_y$ (the sum of all $\delta_y$ now is bounded by $\alpha\varepsilon$, since we scaled everything by $\alpha$). A supervisor controls the total weight used by all $P_y$, and as soon as it reaches $\alpha$, terminates all $P_y$. When the process $P_y$ reaches the weight $\alpha\varepsilon\eta$, it increases the weight of object $i$ by $\alpha\varepsilon\eta$ (here $(i,\eta)$ is the request at node $y$). So everything is as before, but scaled by $\alpha$ and restricted to the subtree rooted at~$x$.
\smallskip

What does this strategy guarantee?
\begin{itemize}
\item The total weight on lengths used by it is at most $\alpha$.
\item The total weight on objects does not exceed the total weight on lengths.
\item If the limit path  $a$ exists and goes through $x$, then either
   \begin{itemize}
   \item the strategy halts and the opponent either fails to match all the weights or spends more than $c\alpha$ on the subtree rooted at~$x$; or
   \item the strategy does not halt, and the semimeasure on objects generated by this strategy $*$-ex\-ceeds $\m^a$, if we omit from $\m^a$ all the requests on the path to $x$.
   \end{itemize}
\end{itemize}
The argument is the same as for the full tree: if the limit path exists and the strategy does not halt, then all the requests along the limit path (except for finitely many of them below $x$) are fulfilled with coefficient $\alpha\varepsilon$. If the strategy halts, the weight used along the limit path does not exceed $\alpha\varepsilon$ (since the sum of requests along each path is bounded by $1$). The weight used in the other nodes of the $x$-subtree is at least $\alpha(1-2\varepsilon)$, including at least $\alpha(1-3\varepsilon)$ matched weight that should be doubled along the limit path, and we achieve the desired goal for $c=2-6\varepsilon$.

\begin{remark}
 In the statement above we have to change $\m^a$ by deleting the requests on the path to $x$. We can change the construction by moving requests up the tree when processing node $x$ to get rid of this problem. One may also note that omitted requests deal only with finitely many objects, so one can average the resulting semimeasure with some semimeasure that is positive everywhere. So we may ignore this problem in the sequel.
\end{remark}

\subsubsection*{How to Win the Game for $c<3$.}

Now we make the crucial step: we show how one can increase $c$ by recursively using our strategies. Recall our strategy for $c<2$, and change it in the following ways:
\begin{itemize}
\item Instead of assigning some length $l_x$ for each node $x$, let us assign an infinite (decidable uniformly in $x$) set $L_x$ of integers; all elements should be greater than $|x|$ (the length of $x$) and for different $x$ these sets should be disjoint (this is easy to achieve).
\item We agree that process $P_x$ (to be defined) uses only lengths from $L_x$.
\item As before $P_x$ is active when $x$ is on the current path, and sleeps otherwise.
\item Previously $P_x$ increased the weight of $l_x$ in small portions, and after each small increase waited until the opponent matched this increase along the current path. Now, instead of that, $P_x$ calls the $x$-strategy described in the previous section, with small $\alpha=\delta_x$, waits until this strategy terminates forcing the opponent to spend almost $2\delta_x$, then calls another instance of the $x$-strategy, waits until it terminates, and so on. For this, $P_x$ divides $L_x$ into infinite subsets $L_x^1+L_x^2+\ldots$, using $L_x^s$ for the $s$th call of an~$x$-strategy, and using $\delta_x$ as the budget for each call. 
\end{itemize}

There are several possibilities for the behavior of an $x$-strategy called recursively. It may happen that it runs indefinitely. This happens when $x$ is an initial segment of the limit path, the $x$-strategy never exceeds its budget $\delta_x$, and the global strategy does not come close to $1$ in its total spending. It this case we win the game, since the part of the semimeasure on objects built by the $x$-strategy is enough to $*$-exceed $\m^a$. This case is called ``the golden run'' in the original exposition of the proof.

If $x$ is not on the limit path, the execution of the $x$-strategy may be interrupted; in this case we only know that it spent not more than its budget, and that the weight used for objects does not exceed the weight used for lengths. This is similar to the case when an increase at $l_x$ was not matched because the path changed.

The $x$-strategy may also terminate. In this case we know that the opponent used almost twice the budget ($\delta_x$) on the extensions of $x$, and a new call of the $x$-strategy is made for another set of lengths. This is similar to the case when the increase at $l_x$ was matched; the advantage is that now the opponent used almost twice our weight.

Finally, the strategy may be interrupted because the total weight used by $x$-processes for all $x$ came close to $1$. After that everything stops, and we just wait until the opponent will be unable to match all the existing weights or forced to use total weight close to $3$. Indeed, most of our weight, except for $O(\varepsilon)$, was used not on the limit path and already matched with factor close to $2$ there --- so matching it again on the limit path makes the total weight close to $3$. 

\subsubsection*{Induction Step.}

Now it is clear how one can continue this reasoning and construct a winning strategy for arbitrary $c$. To get a strategy for some $c$,  we follow the described scheme, and the process $P_x$ makes sequential recursive calls of $c'$-strategies for smaller $c'$. We need $c-c'<1$, so let us use $c'=c-0.5$. More formally, we recursively define a process $S(c, x,\alpha,L)$ where $c$ is the desired amplification, $x$ is a node, $\alpha$ is a positive rational number (the budget), and $L$ is an infinite set of integers greater than $|x|$.\footnote{The pedantic reader may complain that the parameter is an infinite set. It is in fact enough to consider infinite sets from some class, say, decidable sets (as we noted in the previous section), or just arithmetic progressions. Such sets are enough for our purposes and have finite representation. Indeed, an arithmetic progression can be split into countably many arithmetic progressions. For example, $1,2,3,4,\ldots$ can be split into $1,3,5,7,\ldots$ (odd numbers), $2,6,10,14\ldots$ (odd numbers times $2$), $4,12,20,28,\ldots$ (odd numbers times $4$), etc.} The requirements for $S(c,x,\alpha,L)$:

\begin{itemize}
\item It increases only weights of lengths in $L$.
\item The total weight used for lengths does not exceed~$\alpha$.
\item At each step the total weight used for objects does not exceed the total weight used for lengths.
\item Assuming that the process is not terminated externally (this means that $x$ belongs to the current path, starting from some moment), it may halt or not, and:
\begin{itemize}
    \item If the process halts, the opponent uses more that $c\alpha$ on strings that have length in $L$ and are above $x$.
    \item If the process does not halt and the limit path $a$ exists, the part of the semimeasure on objects generated by this process alone is enough to $*$-exceed $\m^a$.
\end{itemize}
\end{itemize}

The implementation of $S(c,x,\alpha,L)$ uses recursive calls of $S(c-0.5,y,\beta,L')$; for each $y$ above $x$ a sequence of those calls is made with $\beta=\delta_y$ and sets $L'$ that are disjoint subsets of $L$ (for different $y$ these $L'$ are also disjoint), similar to what we have described above for the case $c<3$.  \qed
\end{proof}

\section{$\KPt$-low and $\MLR$-low Oracles}

In this section we present one more characterization of $\KPt$-low (or $\KPt$-trivial) sequences: this class coincides with the class of sequences that (being used as oracles) do not change the notion of Martin-L\"of randomness. As almost all notions of computability theory, the notion of Martin-L\"of randomness can be relativized to an oracle $a$; this means that the algorithms that enumerate Martin-L\"of tests now may use the oracle $a$. In this way we get (in general) a wider class of effectively null sets, and therefore fewer but more pronouncedly random sequences. However, for some~$a$, relativizing to $a$ leaves the class of Martin-L\"of random sequences unchanged.

\begin{definition}
A sequence $a$ is $\MLR$-low if every Martin-L\"of random sequence is Martin-L\"of random relative to the oracle $a$.
\end{definition}

The Schnorr--Levin criterion of randomness in terms of prefix complexity shows that if $a$ is $\KPt$-low, then $a$ is also $\MLR$-low. The other implication is also true but more difficult to prove.

\begin{theorem}\label{thm:MLR-K-low}
Every $\MLR$-low sequence is $\KPt$-low.
\end{theorem} 

We will prove a more general result, but first let us give the definitions.

\begin{definition}
Let $a$ and $b$ be two sequences, considered as oracles. We say that $a\leK b$ if $$\KP^b(x)\le \KP^a(x)+O(1).$$ We say that $a\leMLR b$ if every sequence that is Martin-L\"of random relative to~$b$ is also Martin-L\"of random relative to~$a$. 
\end{definition}

If one oracle $b$ is stronger in the Turing sense than another oracle $a$, then~$b$ allows to generate a larger class of effectively null sets, and the set of random sequences relative to $b$ is smaller that the set of random sequences relative $a$; the Kolmogorov complexity function relative to~$b$ is also smaller than Kolmogorov complexity function relative to~$a$. Therefore, we have $a\leMLR b$ and $a\leK b$. So both orderings are coarser than the Turing degree ordering.

We can now reformulate the definitions of $\KPt$-lowness and $\MLR$-lowness: a sequence $a$ is $\KPt$-low if $a\leK \mathbf{0}$ and is $\MLR$-low if $a\leMLR \mathbf{0}$. So to prove Theorem~\ref{thm:MLR-K-low} it is enough to prove the following result~\cite{khms}:

\begin{theorem}\label{thm:K-MLR}
The conditions $a\leK  b$ and  $a\leMLR b$ are equivalent.
\end{theorem}

\begin{proof}
The left-to-right direction once again follows directly from the Schnorr--Levin randomness criterion. The proof in the other direction is more difficult\footnote{This was to be expected. The relation $a\leK b$ is quantitative: it states that two functions coincide with $O(1)$-precision; whether the relation $a\leMLR b$ holds, on the other hand, is a qualitative yes/no question. One can also consider the quantitative version, with randomness deficiencies, but this is unnecessary: the relation $\leMLR$ is already strong enough to obtain an equivalence.}, and will be split in several steps.

Recall that the set of non-random sequences (in the Martin-L\"of sense; we do not use other notions of randomness here) can be described using a universal Martin-L\"of test, that is, represented as the intersection of effectively open sets
     $$
 U_1\supset U_2 \supset U_3\supset\ldots     
     $$
where $U_i$ has measure at most $2^{-i}$ for all $i$. The following observation goes back to Ku\v cera and says that  the first layer of this test,  the set $U_1$, is enough to characterize all non-random sequences. 
\begin{lemma}\label{kucera}
Let $U$ be an effectively open set of measure less than $1$ that contains all non-random sequences. Then a sequence $x=x_0 x_1 x_2\ldots$ is non-random if and only if all its tails $x_k x_{k+1} x_{k+2}\ldots$ belong to $U$. 
\end{lemma}

\begin{proof}
If $x$ is non-random, then all its tails are non-random and therefore belong to $U$. For the other direction we represent $U$ as the union of disjoint intervals $[u_0], [u_1],\ldots$ (by $[v]$ we denote the set of all sequences that have prefix~$v$). Their total measure $\rho=\sum 2^{-|u_i|}$ is less than $1$. If all tails of $x$, including $x$ itself, belong to $U$, then $x$ starts with some $u_i$. The rest is a tail that starts with some $u_j$, etc., so $x$ can be split into pieces that belong to $\{u_0,u_1,\ldots\}$. The set of sequences of the form ``some $u_i$, then something'' has measure $\rho$, the set of sequences of the form ``some $u_i$, some $u_j$, then something'' has measure $\rho^2$, etc. These sets are effectively open and their measures $\rho^n$ effectively converge to $0$. So their intersection is an effectively null set and $x$ is non-random. \qed
\end{proof}

The argument gives also the following:

\begin{corollary}
A sequence $x$ is non-random if there exists an effectively open set~$U$ of measure less than $1$ such that all tails of $x$ belong to $U$.
\end{corollary}

\smallskip

This corollary can be relativized, so randomness with oracle $a$ can be characterized in terms of $a$-effectively open sets of measure less that $1$: \emph{a sequence $x$ is $a$-nonrandom if there exists an $a$-effectively open set~$U$ of measure less than $1$ such that all tails of $x$ belong to $U$}. This gives one implication in the following equivalence (here we denote the oracles by capitals letter to distinguish them from sequences):

\begin{lemma}\label{nonfull}
Let $A$ and $B$ be two oracles. Then $A\leMLR B$ if and only if every $A$-effectively open set of measure less than $1$ can be covered by some $B$-effectively open set of measure less than $1$.
\end{lemma}

\begin{proof}
The ``if'' direction ($\Leftarrow$)  follows from the above discussion: if $x$ is not $A$-random, its tails can be covered by some $A$-effectively open set of measure less than $1$ and therefore by some $B$-effectively open set of measure less than $1$, so $x$ is not $B$-random.

In the other direction: assume that $U$ is an $A$-effectively open set of measure less than~$1$ that cannot be covered by any $B$-effectively open set of measure less than $1$. The set $U$ is the union of an $A$-enumerable sequence of disjoint intervals $[u_1],[u_2],[u_3]$, etc. Consider a set $V$ that is $B$-effectively open, contains all $B$-non-random sequences and has measure less than $1$ (e.g., the first level of the universal $B$-Martin-L\"of test). By assumption $U$ is not covered by $V$, so some interval $[u_i]$ of $U$ is not entirely covered by $V$.

The set $V$ has the following special property: if it does not contain \emph{all} points of some interval, then it cannot contain \emph{almost all} points of this interval, i.e.,  the uncovered part must have some positive measure. Indeed, the uncovered part is a $B$-effectively closed set, and if it has measure zero, it has $B$-effectively measure zero, so all non-covered sequences are $B$-non-random, and therefore should be covered by $V$.

So we found an interval $[u_i]$ in $U$ such that $[u_i]\setminus V$ has positive measure. Then consider the set $V_1=V/u_i$, i.e., the set of infinite sequences $\alpha$ such that $u_i\alpha\in V$. This is a $B$-effectively open set of measure less than $1$, so it does not cover $U$ (again by our assumption). So there exists some interval $[u_j]$ not covered by $V/u_i$. This means that $[u_i u_j]$ is not covered by $V$. We repeat the argument and conclude that the uncovered part has positive measure, so $V/u_i u_j$ is a $B$-effectively open set of measure less than $1$, so it does not cover some $[u_k]$, etc.  In the limit we obtain a sequence $u_i u_j u_k\ldots$ whose prefixes define intervals not covered fully by $V$. Since $V$ is open, this sequence does not belong to $V$, so it is $B$-random. On the other hand, it is not $A$-random, as the argument from the proof of Lemma~\ref{kucera} shows. 
      \qed
\end{proof}

Let us summarize how far we have come so far. Assuming that $A\leMLR B$, we have shown that every $A$-effectively open set of measure less than $1$ can be covered by some $B$-effectively open set of measure less than $1$. What we need to show is that $A\leK B$, i.e., $\KP^B\le \KP^A$ (up to an additive constant), or $\m^A\le \m^B$ (up to a constant factor). This can be reformulated as follows: \emph{for every lower $A$-semicomputable converging series $\sum a_n$ of reals there exists a converging lower $B$-semicomputable series $\sum b_n$ of reals such that $a_n\le b_n$ for every $n$}.

So to connect our assumption and our goal, we need to find a way to convert a converging lower semicomputable series into an effectively open set of measure less than $1$ and vice versa. We may assume without loss of generality that all $a_i$ are strictly less than $1$. Then $\sum a_n<\infty$ is
equivalent to 
    $$
(1-a_0)(1-a_1)(1-a_2)\ldots > 0.    
    $$
This product is a measure of an $A$-effectively closed set 
    $$
 [a_0,1]\times [a_1,1]\times[a_2,1]\times\ldots    
    $$
whose complement 
  $$
  U= \{ (x_0,x_1,\ldots)\mid (x_0<a_0) \lor (x_1<a_1)\lor \ldots \}   
    $$
is an $A$-effectively open set of measure less than $1$. (Here we split Cantor space into a countable product of Cantor spaces and identify each of them with $[0,1]$ equipped with the standard uniform measure on the unit interval.) We are finally ready to apply our assumption and find some $B$-effectively open set $V$ that contains~$U$.
  
Let us define $b_0$ as the supremum of all $z$ such that
 $$[0,z]\times [0,1]\times [0,1]\times\ldots \subset  V$$   
This product is compact for every $z$, and $V$ is $B$-effectively open, so we can $B$-enumerate all rational $z$ with this property, and their supremum $b_0$ is lower $B$-semicomputable. Note that all $z<a_0$ have this property (the set $[0,a_0)\times [0,1]\times[0,1]\times\ldots$ is covered by $U$), so $a_0\le b_0$. In a similar way we define all $b_i$ and get a lower $B$-semicomputable series $b_i$ such that $a_i\le b_i$. It remains to show that $\sum b_i$ is finite. Indeed, the set
    $$
 \{ (x_0,x_1,\ldots)\mid (x_0<b_0) \lor (x_1<b_1)\lor \ldots \}   
    $$
is a part of $V$, and therefore has measure less than $1$; its complement 
    $$
 [b_0,1]\times [b_1,1]\times[b_2,1]\times\ldots    
    $$   
has measure $(1-b_0)(1-b_1) (1-b_2)\ldots$, therefore this product is positive and the series $\sum b_i$ converges. This finishes the proof.  
    \qed 
\end{proof}

\subsubsection*{Acknowledgments.}
This exposition was finished while one of the authors (A.S.) was invited to the National University of Singapore IMS's \emph{Algorithmic Randomness} program. The authors thank the IMS for the support and for the possibility to discuss several topics (including this exposition) with the other participants (special thanks to Rupert H\"olzl and Nikolai Vereshchagin; R.H. also kindly looked at the final version and corrected many errors there). We also thank Andr\'e Nies for the suggestion to write down this exposition for the Logic Blog and for his comments. We have very useful discussions with the participants of the course taught by L.B. at the Laboratoire Poncelet (CNRS, Moscow), especially Misha Andreev and Gleb Novikov. Last but not least, we are grateful to Joe Miller who presented the proof of the equivalence between $\leK$ and $\leMLR$ while visiting the LIRMM several years ago.


\begin{thebibliography}{99}

\bibitem{dhns}
Downey, R.,  Hirschfeldt, D.,  Nies, A., and Stephan F., Trivial reals. In: R. Downey, D. Decheng, T. S. Ping, Q. Y. Hui, and M. Yasugi (eds.), \emph{Proceedings of the 7th and 8th Asian Logic Conferences}. Singapore University Press and World Scientific, 2003. P.~103--131.

\bibitem{dh}
Downey, R., Hirschfeldt, D. 
\emph{Algorithmic Randomness and Complexity}, Springer, 2010.

\bibitem{hs}
Hirshfeldt, D., Nies, A., Stephan, F., Using random sets as oracles. \emph{Journal of the London Mathematical Society}, \textbf{75}, 610--622 (2007).

\bibitem{khms}
 Kjos-Hanssen, B., Miller, J.,  Solomon, R.,   Lowness notions, measure and domination, \emph{Journal of the London Mathematical Society},  \textbf{85}(3), 869--888 (2012).

\bibitem{nies-main}
Nies, A., Lowness properties and randomness, \emph{Advances in Mathematics}, \textbf{197}, 274--305 (2005).

\bibitem{nies}
Nies, A., \emph{Computability and Randomness}, Oxford University Press, 2009. 

\bibitem{uppsala-notes}
Shen, A., Algorithmic information theory and Kolmogorov complexity, lecture notes, Uppsala University Technical Report TR2000-034, \url{http://www.it.uu.se/research/publications/reports/2000-034/2000-034-nc.ps.gz}

\bibitem{usv}
Shen, A., Uspensky, V., Vereshchagin, N., Kolmogorov complexity and algorithmic randomness, MCCME, 2012 (Russian). Draft translation: \url{www.lirmm.fr/~ashen/kolmbook-eng.pdf}.
\end{thebibliography}
\end{document}